\documentclass[11pt,twoside]{article}
\usepackage{authblk}
\usepackage[scaled=0.8]{DejaVuSansMono}
\usepackage{emptypage}
\usepackage[a4paper,marginratio={5:5,5:7},width=147mm]{geometry}
\usepackage[T2A,T1]{fontenc}
\usepackage[utf8x]{inputenc}

\usepackage{fancyhdr}
\newcommand{\maybebf}{}

\fancypagestyle{plain}{%
    \fancyhf{}%
    \fancyfoot[C]{\maybebf \thepage}
}
\fancypagestyle{intro}{
    \renewcommand{\sectionmark}[1]{\markboth{##1}{}}
    \fancyhf{}
    \fancyhead[LE,RO]{\leftmark}
    \fancyfoot[C]{\maybebf \thepage}

}

\renewcommand{\sectionmark}[1]{%
  \markboth{%
    \ifnum\value{section}>0
      \maybebf{\thesection.}\space
    \fi
    #1%
  }{%
    \ifnum\value{subsection}=0
            \thesection. #1
        \fi%
  }%
}
\renewcommand{\subsectionmark}[1]{%
    \markright{%
        \ifnum\value{subsection}>0
            \thesubsection. #1
        \fi%
    }%
}
\usepackage[english]{babel}
\usepackage{graphicx}
\usepackage{amsmath}
\usepackage{amssymb}
\usepackage{amsthm}
\usepackage{mathtools}
\usepackage[final]{microtype}
\usepackage{tikz}
\usetikzlibrary{arrows.meta,arrows,decorations.pathmorphing,shapes,knots,hobby,quotes,positioning}%,external}
\usepackage{asymptote}
\usepackage{tikz-cd}
\usepackage{parskip}
\usepackage{float}
\usepackage{caption}
\usepackage{subcaption}
\usepackage{enumitem}
\usepackage{adjustbox}
\usepackage{xfakebold}
\usepackage{mdframed}
\usepackage{booktabs}
\usepackage{multirow}
\usepackage{siunitx}
\usepackage{nicematrix}
\NiceMatrixOptions{xdots/shorten=0.87em, renew-dots}
\usepackage{diagbox}
\usepackage[calc,showdow,english]{datetime2}
\DTMnewdatestyle{mydateformat}{%
}

\usepackage{comment}
\usepackage{import}
\usepackage{standalone}
\usepackage{xcolor}

\definecolor{zg}{gray}{0.6}
\newcommand{\zg}[1]{\textcolor{zg}{#1}}

\usepackage{csquotes}
\usepackage[doi=false,isbn=false,
maxbibnames=99,backend=biber,style=numeric,sorting=nyt]{biblatex}
\usepackage[unicode]{hyperref}
\usepackage[hyphenbreaks]{breakurl}

\definecolor{hlinkcol}{HTML}{A00000}
\definecolor{hcitecol}{HTML}{308030}
\hypersetup{
    colorlinks,
    linkcolor={teal!100!black},
    citecolor={green!50!black},
    urlcolor={blue!0!black}
}

\theoremstyle{plain}
\newtheorem{thm}{Theorem}[section]
\newtheorem{prop}[thm]{Proposition}
\newtheorem{lemma}[thm]{Lemma}

\newtheorem*{thm*}{Theorem}
\newtheorem*{thm20*}{Theorem \ref*{thm:main20}}
\newtheorem*{lemma*}{Lemma}
\newtheorem*{prop*}{Proposition}
\newtheorem*{conj*}{Conjecture}
\newtheorem*{cor*}{Corollary}

\theoremstyle{definition}
\newtheorem{defin}[thm]{Definition}

\newtheorem{question}[thm]{Question}
\newtheorem*{defin*}{Definition}

\theoremstyle{remark}
\newtheorem{remark}[thm]{Remark}

\counterwithin{figure}{section}
\counterwithin{table}{section}

\newcommand{\inj}{\hookrightarrow}
\newcommand{\surj}{\twoheadrightarrow}

\newcommand{\ZZ}{\mathbb Z}
\newcommand{\QQ}{\mathbb Q}

\newcommand{\PP}{\mathbb P}
\newcommand{\RP}{\mathbb{RP}}
\newcommand{\ww}{{w}}
\newcommand{\vv}{{v}}
\newcommand{\vvd}{\overline{v}}

\tikzset{
dot/.style = {circle, fill, minimum size=#1,
              inner sep=0pt, outer sep=0pt},
dot/.default = 5pt
}

\DeclarePairedDelimiter\abs{\lvert}{\rvert}%
\DeclarePairedDelimiter\norm{\lVert}{\rVert}%

\makeatletter
\let\oldabs\abs
\def\abs{\@ifstar{\oldabs}{\oldabs*}}
\let\oldnorm\norm
\def\norm{\@ifstar{\oldnorm}{\oldnorm*}}
\makeatother

\DeclareMathOperator{\diag}{diag}

\newcommand{\fbseries}{
\unskip\setBold\aftergroup\unsetBold\aftergroup\ignorespaces}
\makeatletter
\newcommand{\setBoldness}[1]{\def\fake@bold{#1}}
\makeatother

\tikzset{
    rots/.style={anchor=south, rotate=90, inner sep=.5mm}
}

\newcommand\restr[2]{{%
  \left.\kern-\nulldelimiterspace %
  #1 %
  \vphantom{\big|} %
  \right|_{#2} %
  }}

\author{Jacopo G. Chen\thanks{Scuola Normale Superiore, Pisa, Italy. Email: \href{mailto:jacopo.chen@sns.it}{\texttt{jacopo.chen@sns.it}}}}
\title{Non-cobordant hyperbolic manifolds}

\begin{document}

\DTMsetdatestyle{mydateformat}
\date{}
\maketitle
%\vspace{-2ex}
\begingroup
\centering\small
\textbf{Abstract}\par\smallskip
\begin{minipage}{\dimexpr\paperwidth-9.5cm}
In all dimensions $n \ge 4$ not of the form $4m+3$, we show that there exists a closed hyperbolic $n$-manifold which is not the boundary of a compact $(n+1)$-manifold. 
The proof relies on the relationship between the cobordism class and the fixed point set of an involution on the manifold, together with a geodesic embedding of Kolpakov, Reid and Slavich.
We also outline a possible approach to cover the dimensions $4m+3 \ne 2^k-1$.
\end{minipage}
\par\endgroup
\section{Introduction}\label{sec:intro}
The (unoriented) cobordism relation on smooth, closed $n$-manifolds gives rise to the \emph{cobordism group} $\mathcal N_n$, whose elements are equivalence classes of cobordant manifolds, and whose group operation is induced by the disjoint union: $[M] + [M'] \coloneqq [M \sqcup M']$. Moreover, the direct sum of all $\mathcal N_n$ for $n \ge 0$ can be given a graded ring structure, with multiplication induced by the Cartesian product of manifolds. The structure of the \emph{cobordism ring} $\mathcal N_*$ was described by Thom in his seminal paper~\cite{thom}: it is a free polynomial ring over $\ZZ_2$
\begin{equation}
    \mathcal N_* \simeq \ZZ_2[x_2, x_4, x_5, x_6, \dots],
\end{equation}
with one generator $x_i$ for each $i \ne 2^k - 1$. Of course, this allows one to easily determine the individual groups $\mathcal N_n$.

It is also of interest to impose additional structure on manifolds and the cobordisms between them: for instance, we may study the \emph{oriented cobordism} or \emph{spin cobordism} groups. In the context of hyperbolic geometry, Long and Reid~\cite{geometric-bd} introduced the concept of a manifold that \emph{bounds geometrically}, that is, a hyperbolic manifold that is the geodesic boundary of another hyperbolic manifold.

Returning to unoriented cobordism, we may instead impose conditions on the Riemannian metric: it is a classical result of Hamrick and Roysted~\cite{flat-bnd} that all flat closed manifolds are boundaries; on the other hand, for even $n$, the real projective spaces $\RP^n$ are examples of non-cobordant spherical manifolds, since their Euler characteristic is $1$. Moreover, by Ontaneda's Riemannian hyperbolization~\cite[Corollary~1]{ontaneda}, for any given $\varepsilon>0$, we can realize \emph{every} cobordism class with closed manifolds of pinched negative sectional curvature $K \in [-1-\varepsilon, -1]$.

Along this direction, in the present paper we will analyze the unoriented cobordism relation between closed hyperbolic manifolds. Specifically, we prove:
\begin{thm}\label{thm:main20}
    For each $n \ge 4$, $n \not \equiv 3 \pmod 4$, there exists a connected, non-cobordant closed hyperbolic $n$-manifold.
\end{thm}
Note that the two-dimensional case is realized by any hyperbolic surface diffeomorphic to $\mathbb T^2 \# \RP^2$, while the case $n = 4$ is realized by two manifolds of Euler characteristic $17$, due to Ratcliffe and Tschantz~\cite{rt-instantons}. We have not been able to find any other examples of non-cobordant closed hyperbolic manifolds in the literature.

By contrast, we may arrange for our examples to have even Euler characteristic, which implies that in dimension $4$ they realize a different cobordism class from either Ratcliffe--Tschantz $4$-manifold (see Section~\ref{sec:even}).

The proof of Theorem~\ref{thm:main20} relies on describing the cobordism class of a manifold $M$ in terms of \emph{projective bundles} on the fixed submanifolds $X$ of a nontrivial involution $\tau\colon M \to M$. By a formula of Thom, we can compute the value of a certain homomorphism $\varphi^n \colon \mathcal N_n \surj \ZZ_2$ on such a projective bundle, discovering that it depends only on the Stiefel--Whitney classes of the normal bundle $\nu X$ (and not on the tangent classes $w_i(X)$). 
The formula, stated in Theorem~\ref{thm:phi-n-sum-Ind}, may be of independent interest. It takes the following form:
 \begin{equation}\label{eq:formula-redux}
    \varphi^n[M] = \sum_{d = 0}^{n-1} \sum_{X \in F_d(\tau)} I_{n,d}(w_1(\nu X), \dots, w_d(\nu X)),
\end{equation}
where $F_d(\tau)$ is the set of all fixed submanifolds of $\tau$ having dimension $d$, and $I_{n,d}$ are certain polynomial expressions.

In Section~\ref{sec:geod}, we build upon a result of Kolpakov, Reid and Slavich~\cite{embedding}, which roughly gives a way to geodesically embed an arithmetic hyperbolic $n$-manifold $M$ into an arithmetic hyperbolic $(n+1)$-manifold $M'$; specifically, we extend the result to the context of arithmetic manifolds with involutions. Moreover, we show that by cutting $M'$ along $M$ and reattaching with a twist given by the involution, we obtain another $(n+1)$-manifold $M''$ in which $M$ also embeds geodesically.

The right hand side of~(\ref{eq:formula-redux}) is a well-defined invariant $I(n, M, \tau)$ for a manifold with involution $(M, \tau)$ of any dimension, coinciding with $\varphi^n[M]$ when $\dim M = n$: using this fact, we construct a non-cobordant $n$-manifold $M^n$ by starting from a particular $3$-manifold with $I(n, M, \tau) = 1$, and repeatedly applying the Kolpakov--Reid--Slavich embedding to increase the dimension. At each step, we may choose whether to perform the twist or not; it is exactly this freedom that allows us to ensure $I(n, M, \tau) = 1$ at each step, until we reach dimension $n$.

Furthermore, for $3$-manifolds, the expression $I(n, M, \tau)$ is periodic in $n$ of period $4$, so we manage to cover all dimensions $n \equiv 0,1,2 \pmod 4$ with only three starting examples. The limitations of our method for the remaining case $n \equiv 3 \pmod 4$, which we discuss in Section~\ref{sec:higher}, arise from a specific pattern in the expressions $I_{n,d}$. By choosing suitable starting manifolds, we describe how one might construct non-cobordant hyperbolic manifolds in all dimensions $n \ne 2^k-1$. Note that the cobordism groups $\mathcal N_{2^k-1}$ are actually nontrivial for $k \ge 3$, but constructing non-cobordant hyperbolic manifolds in these dimensions requires a different approach.

Some questions remain open:
\begin{question}
    Do there exist non-cobordant hyperbolic manifolds of dimension $4m+3$?
\end{question}
\begin{question}
    Which cobordism classes can be realized by hyperbolic manifolds? What if the manifolds are required to be connected?
\end{question}
   
Additionally, since our method does not easily allow for the construction of orientable manifolds, we may ask:
\begin{question}
    Do there exist non-cobordant orientable hyperbolic manifolds?
\end{question}
In dimensions $2$ and $4$ the answer is known to be negative. Indeed, an orientable surface bounds a handlebody; on the other hand, a closed orientable hyperbolic $4$-manifold has zero signature by the Hirzebruch signature formula~\cite{geometric-bd}, thus it is cobordant.

Similar questions may of course be asked for the orientable cobordism group and its higher analogs. 

Finally, a natural question arises regarding $\text{spin}^c$ structures on hyperbolic manifolds. In a previous paper~\cite{spinc}, the author showed that, in all dimensions $\ge 5$, there exists a closed orientable hyperbolic manifold $M$ with $w_3(M) \ne 0$, which in turn implies that $M$ does not admit a $\text{spin}^c$ structure. The argument also relies on iterating the Kolpakov--Reid--Slavich embedding, but is simpler than the one in the present paper, in that it directly keeps track of nonzero Stiefel--Whitney classes at each step. Now note that, if a $5$-manifold $M$ is not cobordant, then its Stiefel--Whitney number $w_2w_3$ is nontrivial, implying $w_3(M) \ne 0$. Hence, a stronger version of Theorem~\ref{thm:main20}, giving orientable manifolds, would imply the main result of~\cite{spinc} in dimension $5$. It would be interesting to generalize this argument to higher dimensions:
\begin{question}
    Can the proof strategy of Theorem~\ref{thm:main20} be adapted to show the existence of non-$\text{spin}^c$ closed hyperbolic manifolds via the nonvanishing of Stiefel--Whitney numbers?
\end{question}
\subsection*{Structure of the paper}
In Section~\ref{sec:cobord}, we define the unoriented cobordism ring $\mathcal N_*$, its decomposable classes and the homomorphisms $\varphi^n$. Then, in Section~\ref{sec:proj-inv}, we relate the cobordism class of a manifold to the fixed point set of an involution defined on it, and we compute its image through $\varphi^n$. In Section~\ref{sec:geod}, we introduce the Kolpakov--Reid--Slavich codimension-$1$ embedding, extend its definition to manifolds with involutions, and define an alternate, twisted embedding. Finally, in Section~\ref{sec:non-cobord}, we prove Theorem~\ref{thm:main20} and discuss the remaining case $n\equiv 3 \pmod 4$.
\subsection*{Acknowledgments}
I would like to thank my advisor Bruno Martelli for bringing my attention to this problem and for his support during the writing of this paper.
%\pagebreak[4]
\begin{comment}
\renewcommand{\baselinestretch}{1.0}\normalsize
\tableofcontents
\renewcommand{\baselinestretch}{1.0}\normalsize
\clearpage
\end{comment}
\section{The unoriented cobordism ring}\label{sec:cobord}
In this section, we recall some properties of the unoriented cobordism (henceforth simply \emph{cobordism}) relation on manifolds.

\begin{defin}
    Let $M$, $M'$ be two closed $n$-manifolds. A \emph{cobordism} between $M$ and $M'$ is a compact $(n+1)$-manifold with boundary $W$ such that $\partial W \simeq M \sqcup M'$. If such a $W$ exists, we say that $M$ and $M'$ are \emph{cobordant} and have the same \emph{cobordism class} $[M] = [M']$. Finally, we say that $M$ is \emph{cobordant} if $[M] = [\varnothing]$.
\end{defin}
Cobordism classes of $n$-manifolds form a group $\mathcal N_n$ under disjoint union, with $[\varnothing]$ as the identity; moreover, every class is its own inverse, because $\partial(M \times [0,1]) \simeq M \sqcup M$, making $\mathcal N_n$ into an abelian $2$-group.

The Cartesian product of manifolds induces a well-defined multiplication of cobordism classes, and hence a graded ring structure on the direct sum of all cobordism groups
\begin{equation}
    \mathcal N_* \coloneqq \sum_{i \ge 0} \mathcal N_i.
\end{equation}

The structure of the cobordism ring is that of a free polynomial algebra on $\ZZ_2$, as proved by Thom:
\begin{thm}[{\cite[Théorème~IV.12]{thom}}]
    We have
    \begin{equation}
        \mathcal N_* \simeq \ZZ_2[x_2, x_4, x_5, x_6, \dots],
    \end{equation}
    with a generator $x_i$ of degree $i$ for each $i \ne 2^k-1$.
\end{thm}
Thom also showed~\cite[80]{thom} that the real projective spaces $\RP^i$ can be taken as representatives for $x_i$, $i$ even, while Dold~\cite{dold} introduced the manifolds $P(m,n)$ as representatives for $x_i$, $i$ odd:
\begin{thm}[{\cite[Satz~3]{dold}}]
    Representatives for each generator $x_i$, $i \ne 2^k-1$, can be chosen as follows. If $i$ is even, then $x_i = [\RP^i]$. If $i$ is odd, write $i = 2^r(2s+1)-1$; then $x_i = [P(2^r-1, s2^r)]$, where
    \begin{equation}
        P(m, n) \coloneqq (S^m \times \mathbb{CP}^n)/\{(x,[y]) \sim (-x, [\overline y])\}.
    \end{equation}
\end{thm}

\subsection{Decomposable classes}\label{sec:decompos}
In order to simplify the problem of constructing non-cobordant hyperbolic $n$-manifolds, we define nontrivial homomorphisms $\varphi^n\colon \mathcal N_n \surj \ZZ_2$, and then proceed to find a hyperbolic $M^n$ such that $\varphi^n[M^n] = 1$.

\begin{defin}
    A class $y \in \mathcal N_*$ is \emph{decomposable} if it is represented by a product of two manifolds of positive dimension.
\end{defin}
Decomposable classes are sums of products of two or more generators. In every dimension $n \ne 2^k-1$, the group $\mathcal N_n$ is generated by decomposable classes together with $x_n$: hence, it is natural to define a map $\varphi^n$ that sends $x_n$ to $1$ and vanishes on decomposable classes.

The value $\varphi^n[M^n]$ can be elegantly computed, following~\cite{thom} (see also the expository paper~\cite{huang}). Indeed, given the total Stiefel--Whitney class
\begin{equation}
    w(M^n) \coloneqq 1 + w_1 + \dots + w_n,
\end{equation}
let $\lambda_1, \dots, \lambda_n$ be the formal roots of the polynomial $p(t) \coloneqq t^n + w_1t^{n-1} + \dots + w_n$, obtained by homogenizing $w(M^n)$ with a new degree-$1$ variable $t$. Then we have
\begin{equation}\label{eq:varphi-formula}
    \varphi^n[M^n] = \sum_i \lambda_i^n.
\end{equation}
More rigorously, we observe that the sum of powers in~(\ref{eq:varphi-formula}) can be written in terms of elementary symmetric polynomials in the roots $\lambda_i$, which are simply the coefficients $w_1, \dots, w_n$. As expected, $\varphi^n[M^n]$ is a polynomial expression of degree $n$ in the Stiefel--Whitney classes or, in other words, a sum of Stiefel--Whitney numbers.

\section{Projective bundles and involutions}\label{sec:proj-inv}
Given a rank-$k$ vector bundle $\xi$ over a closed $n$-manifold $M$, we can define the \emph{projective bundle} $\mathbb{P}(\xi)$ as a fiberwise quotient of the unit sphere bundle $S(\xi)$ by the antipodal map. This construction results in a bundle over $M$ with fiber $\mathbb{RP}^{k-1}$, whose total space, which we shall also denote by $\mathbb{P}(\xi)$, is a closed $(n+k-1)$-manifold.

The total Stiefel--Whitney class of $\mathbb{P}(\xi)$ can be described in terms of the Stiefel--Whitney classes of $M$ and $\xi$, as follows.
\begin{thm}[{\cite[Theorem~23.3]{diff-periodic}, \cite[517]{borel-hirzebruch}}]\label{thm:borel-hirzebruch}
    Let $\xi$ be a rank-$k$ vector bundle over a closed $n$-manifold $M$ and let $p\colon \mathbb P(\xi) \twoheadrightarrow M$ be the bundle map. Define the classes $\ww_i \coloneqq p^*(w_i(M))$, $\vv_i \coloneqq p^*(w_i(\xi))$. Then we have
    \begin{equation}\label{eq:cohom-proj}
        H^*(\mathbb P(\xi)) \simeq H^*(M)[c] / (c^k + c^{k-1}\vv_1 + \dots + \vv_k),
    \end{equation}
    where $c \in H^1(\mathbb P(\xi))$ is the characteristic class of the double cover $S(\xi) \twoheadrightarrow \mathbb P(\xi)$. In particular, $H^*(\mathbb P(\xi))$ is a free $H^*(M)$-module with basis $\{1, c, \dots, c^{k-1}\}$.
    
    Moreover, the total Stiefel--Whitney class of the manifold $\mathbb P(\xi)$ is
    \begin{equation}
        w(\mathbb P(\xi)) = (1 + \ww_1 + \dots + \ww_n)\cdot 
        ((1+c)^k + (1+c)^{k-1}\vv_1 + \dots + \vv_k).
    \end{equation}
\end{thm}
As we will see in the following, the cobordism class of a manifold equipped with an involution can be written in terms of projective bundles over the fixed submanifolds of the involution.

\subsection{Fixed points of involutions}\label{sec:involutions}
Let $M$ be a connected closed $n$-manifold equipped with a nontrivial involution $\tau$. Following~\cite[Section~24]{diff-periodic}, the fixed point set $\operatorname{Fix}(\tau)$ is a disjoint union of closed submanifolds of $M$. Hence, we may define $F(\tau)$ as the set of all fixed submanifolds, and $F_i(\tau)$ as the set of all such submanifolds of dimension $i$, for $i \ge 0$; note that $F_i(\tau)$ is trivially empty for $i > n$, and $F_n(\tau)$ is also empty since $\tau$ is nontrivial and $M$ is connected.

The following result is of fundamental importance to our discussion.
\begin{thm}[{\cite[Theorem~24.2]{diff-periodic}}]\label{thm:proj-bundles-fix}
    Let $M$ be a connected closed $n$-manifold and let $\tau$ be a nontrivial involution on $M$. Then we have
    \begin{equation}\label{eq:proj-bundles-fix}
        [M] = \sum_{X \in F(\tau)} [\mathbb P(\nu X \oplus \varepsilon^1)],
    \end{equation}
    where $\varepsilon^1$ denotes the trivial line bundle.
\end{thm}
Note that the Stiefel--Whitney classes of $\nu X \oplus \varepsilon^1$ are the same as those of $\nu X$. Hence, by Theorems~\ref{thm:borel-hirzebruch} and~\ref{thm:proj-bundles-fix}, the cobordism class of $M$ is determined by the Stiefel--Whitney classes of $X$ and $\nu X$ for all fixed submanifolds $X$.

\subsection{Applying \texorpdfstring{$\varphi^n$}{phi\^{}n}}
Let us compute the value of $\varphi^n$ applied to such projective bundles, using Theorem~\ref{thm:borel-hirzebruch} and the method of Section~\ref{sec:decompos}.

Let $X \in F_d(\tau)$ be a fixed submanifold of dimension $d$ and let $w_i \coloneqq w_i(X)$, $v_i \coloneqq w_i(\nu X)$. Then, by Theorem~\ref{thm:borel-hirzebruch}, the total Stiefel--Whitney class of $[\mathbb P(\nu X \oplus \varepsilon^1)]$ is
\begin{equation}
(1 + \ww_1 + \dots + \ww_d)\cdot 
        ((1+c)^{n-d+1} + (1+c)^{n-d}\vv_1 + \dots + (1+c)^{n-2d+1}\vv_d).
\end{equation}
After homogenizing, we obtain the polynomial
\begin{equation}
    p(t) \coloneqq (t^d + t^{d-1}\ww_1 + \dots + \ww_d) 
        ((t+c)^{n-d+1} + (t+c)^{n-d}\vv_1 + \dots + (t+c)^{n-2d+1}\vv_d),
\end{equation}
with formal roots $\lambda_1, \dots, \lambda_{n+1}$.
The number $\varphi^n[\PP(\nu X \oplus \varepsilon^1)]$ can then be computed as the formal sum $\lambda_1^n + \dots + \lambda_{n+1}^n$.

Let $\{\alpha_1, \dots, \alpha_d\}$ and $\{\beta_1, \dots, \beta_{n-d+1}\}$ be, respectively, the formal roots of the two polynomials
\begin{equation}
    \begin{split}
        p_1(t) &\coloneqq t^d + t^{d-1}\ww_1 + \dots + \ww_d, \\
        p_2(t) &\coloneqq t^{n-d+1} + t^{n-d}\vv_1 + \dots + t^{n-2d+1}\vv_d.
    \end{split}
\end{equation}
Since $p(t) = p_1(t)\cdot p_2(t+c)$, we have
\begin{equation}
    (\lambda_1, \dots, \lambda_{n+1}) = (\alpha_1, \dots, \alpha_d, \beta_1 + c, \dots, \beta_{n-d+1} + c).
\end{equation}
Now, note that $\alpha_1^n + \dots + \alpha_d^n$ is an expression of degree $n$ in the classes $\ww_i$, which vanishes since $d < n$. Therefore
\begin{align}
    %\sum_{i = 1}^{n+1}\lambda_i^n
    \varphi^n[\PP(\nu X \oplus \varepsilon^1)]
    &= \sum_{i = 1}^{n-d+1}(\beta_i+c)^n \\
    &= \sum_{j = 0}^d \binom{n}{j} c^{n-j}\sum_{i = 1}^{n-d+1}\beta_i^j, \label{eq:formula-i-phi}
\end{align}
where we discard powers of $\beta_i$ with exponent larger than $d$.
For $j \ge 1$, the sum of $j$-th powers in~(\ref{eq:formula-i-phi}) is a fixed polynomial expression, independent of the number of roots, involving the elementary symmetric polynomials in the $\beta_i$ (that is, the classes $v_1, \dots, v_d$): say $p_j(v_1, \dots, v_d)$. For $j = 0$, it takes the value $n-d+1$. Hence, we have
\begin{equation}\label{eq:phi-n-powersums}
    \varphi^n[\PP(\nu X \oplus \varepsilon^1)]
    = (n-d+1)c^n + \sum_{j = 1}^d \binom{n}{j} c^{n-j} p_j(v_1, \dots, v_d).
\end{equation}
We can simplify the formula further by reducing the powers of $c$ modulo the relation $c^{n-d+1} + c^{n-d}v_1 + \dots + c^{n-2d+1}v_d = 0$. Define the \emph{dual Stiefel--Whitney classes} $\vvd_i$, satisfying the formal relation
\begin{equation}
    \bigg(\sum_{j \ge 0}\vv_j t^j\bigg)\cdot \bigg(\sum_{j \ge 0}\vvd_j t^j\bigg) = 1.
\end{equation}
(Note that we can recursively express each $\vvd_i$ as a polynomial over $v_1, \dots, v_d$.)
We claim that $c^{n-j}$ reduces to $c^{n-d}\vvd_{d-j}$, plus terms of lower degree in $c$ (and higher degree in the $v_i$), which can be discarded. This is obvious for $j = d$.

For $j < d$, by using little-$o$ notation $o(c^k)$ for terms of $c$-degree less than $k$, we have:
\begin{align}
    0&= \bigg(\sum_{k=0}^{d-j-1} c^{d-j-1-k}\vvd_k\bigg)
    \bigg(\sum_{k=0}^d c^{n-d+1-k}v_k\bigg) \\
    &= \bigg(\sum_{k = 0}^{d-j} c^{n-j-k} \sum_{m = 0}^{k} \vvd_m v_{k-m} \bigg) - c^{n-d} \vvd_{d-j} + o(c^{n-d})\\
    &= c^{n-j} + c^{n-d}\vvd_{d-j} + o(c^{n-d}),
\end{align}
as desired. Hence, by reducing the powers of $c$ in~(\ref{eq:phi-n-powersums}), we obtain
\begin{equation}\label{eq:phi-n-powersums2}
    \varphi^n[\PP(\nu X \oplus \varepsilon^1)]
    = c^{n-d}\bigg[(n-d+1)\vvd_d + \sum_{j = 1}^d \binom{n}{j} \vvd_{d-j} p_j(v_1, \dots, v_d)\bigg].
\end{equation}
As an element of $H^n(\PP(\nu X \oplus \varepsilon^1))\simeq \ZZ_2$, this equals the expression in square brackets as an element of $H^d(X) \simeq \ZZ_2$, since $\{1, c, \dots, c^{n-d}\}$ is a basis. Thus, we have proved:

\begin{thm}\label{thm:phi-n-sum-Ind}
    Let $M$ be a connected closed $n$-manifold, $n \ne 2^k-1$, and let $\tau$ be a nontrivial involution on $M$. Then we have
    \begin{equation}
        \varphi^n[M] = \sum_{d = 0}^{n-1} \sum_{X \in F_d(\tau)} I_{n,d}(w_1(\nu X), \dots, w_d(\nu X)),
    \end{equation}
    where $I_{n,d}$ is a polynomial expression in the free variables $v_1, \dots, v_d$, defined for any $n, d \ge 0$ and given by
    \begin{equation}\label{eq:defin-Ind}
        I_{n,d}(v_1, \dots, v_d) \coloneqq (n-d+1)\vvd_d
        + \sum_{j=1}^d \binom{n}{j}\vvd_{d-j}p_j(v_1, \dots v_d).
    \end{equation}
\end{thm}

\begin{remark}\label{rem:periodic}
    By applying the map $\varphi^n$, we have eliminated the dependence on the tangent Stiefel--Whitney classes $w_i(X)$. Moreover, the expression $I_{n,d}$ depends on $n$ only through the coefficients $n-d+1$ and $\binom{n}{j}$, $0 < j \le d$, whose parity is periodic in $n$; a common period is the smallest power of two $2^q$ such that $q > 0$ and $2^q > d$.
\end{remark}
Using the computer algebra system SageMath~\cite{sagemath}, we can easily determine a few values of $I_{n,d}$ (Table~\ref{tab:Ind}).

\begin{table}[ht]
    \centering
    \begin{tabular}{c|ccccc}
    \toprule
    \diagbox{$n$}{$d$} & $0$ & $1$ & $2$ & $3$ & $4$ \\ \midrule
 \zg{$0$} & \zg{$1$} & \zg{$0$} & \zg{$v_1^2 + v_2$} &             \zg{$0$} &  \zg{$v_1^4 + v_1^2v_2 + v_2^2 + v_4$} \\
 \zg{$1$} & \zg{$0$} & \zg{$0$} &       \zg{$v_1^2$} &  \zg{$v_1v_2 + v_3$} &                  \zg{$v_1^4 + v_1v_3$} \\
 $2$ & $1$ & $0$ &         \zg{$v_2$} &         \zg{$v_1^3$} &                     \zg{$v_2^2 + v_4$} \\
 \zg{$3$} & \zg{$0$} & \zg{$0$} &           \zg{$0$} &             \zg{$0$} &                           \zg{$v_1^4$} \\ \midrule
 $4$ & $1$ & $0$ & $v_1^2 + v_2$ &             $0$ &          \zg{$v_1^2v_2 + v_2^2 + v_4$} \\
 $5$ & $0$ & $0$ &       $v_1^2$ &  $v_1v_2 + v_3$ &                          $v_1v_3$ \\
 $6$ & $1$ & $0$ &         $v_2$ &         $v_1^3$ &             $v_1^4 + v_2^2 + v_4$ \\
 \zg{$7$} & \zg{$0$} & \zg{$0$} &           \zg{$0$} &             \zg{$0$} &                               \zg{$0$} \\ \midrule
 $8$ & $1$ & $0$ & $v_1^2 + v_2$ &             $0$ &  $v_1^4 + v_1^2v_2 + v_2^2 + v_4$ \\
 $9$ & $0$ & $0$ &       $v_1^2$ &  $v_1v_2 + v_3$ &                  $v_1^4 + v_1v_3$ \\
$10$ & $1$ & $0$ &         $v_2$ &         $v_1^3$ &                     $v_2^2 + v_4$ \\
$11$ & $0$ & $0$ &           $0$ &             $0$ &                           $v_1^4$ \\ \midrule
$12$ & $1$ & $0$ & $v_1^2 + v_2$ &             $0$ &          $v_1^2v_2 + v_2^2 + v_4$ \\
$13$ & $0$ & $0$ &       $v_1^2$ &  $v_1v_2 + v_3$ &                          $v_1v_3$ \\
$14$ & $1$ & $0$ &         $v_2$ &         $v_1^3$ &             $v_1^4 + v_2^2 + v_4$ \\
\zg{$15$} & \zg{$0$} & \zg{$0$} &           \zg{$0$} &             \zg{$0$} &                               \zg{$0$} \\ \midrule
$16$ & $1$ & $0$ & $v_1^2 + v_2$ &             $0$ &  $v_1^4 + v_1^2v_2 + v_2^2 + v_4$ \\
\bottomrule
    \end{tabular}
    \caption{Values of $I_{n,d}$ for $n \le 16$, $d \le 4$. Cells in gray correspond to invalid combinations ($n \le d$ or $n = 2^k-1$) for which $I_{n,d}$ can still be defined.
    Note the periodicity of the columns.}
    \label{tab:Ind}
\end{table}

A pattern emerges for $n$ of the form $2^k-1$, where exactly the first $2^k$ terms of the sequence $(I_{n,d})_{d\ge 0}$ are zero. Indeed, we can prove:
\begin{prop}\label{prop:initial-zeros}
    Let $k,m > 0$ with $m$ odd. We have
    \begin{itemize}
        \item $I_{m2^k-1, d} = 0$ for all $0 \le d < 2^k$;
        \item $I_{m2^k-1, 2^k} = v_1^{2^k}$.
    \end{itemize}
\end{prop}
\begin{proof}
    A common period of $I_{n, 0}, \dots, I_{n, 2^k}$, as functions of $n$, is $2^{k+1}$. Hence, without loss of generality, we can assume $m = 1$. 

    For $0 \le d \le 2^k$, we have, since $n-d+1 \equiv 2^k + d \equiv d \pmod 2$,
    \begin{equation}\label{eq:I-2k-d}
        I_{2^k-1, d}
        = d\vvd_d + \sum_{j = 1}^d \binom{2^k-1}{j} \vvd_{d-j}p_j(v_1, \dots, v_d).
    \end{equation}
    The graded polynomial ring $\ZZ_2[v_1, v_2, \dots]$ is isomorphic to $\Lambda^{\ZZ_2}$, the \emph{ring of symmetric functions} over $\ZZ_2$, with the isomorphism sending $v_i$ to the $i$-th {elementary symmetric function}. By identifying the two rings, we can interpret $\vvd_j$ and $p_j$ as symmetric functions and exploit some known identities. We refer the reader to~\cite{symm-func}.
    
    The expressions $\vvd_i = \vvd_i(v_1, \dots, v_i)$ satisfy the same recurrence
    \begin{equation}
        \vvd_i = \sum_{j=0}^{i-1} v_{i-j}\vvd_j
    \end{equation}
    as the complete homogeneous symmetric functions $h_i$, so we have $\vvd_i(v_1, \dots, v_i) = h_i(v_1, \dots, v_i)$ as expressions in $v_1, \dots, v_i$. Moreover, by a well-known identity involving $h_j$ and the power sum symmetric functions $p_j$ we have, for each $d \ge 0$:
    \begin{equation}\label{eq:relation-hp}
        0 = dh_d + \sum_{j=1}^d h_{d-j}p_j = d\vvd_d + \sum_{j=1}^d \vvd_{d-j}p_j. 
    \end{equation}
    The right hand sides of~(\ref{eq:I-2k-d}) and~(\ref{eq:relation-hp}) are almost the same: the binomial coefficient is always $1$, except for the case $j = d = 2^k$. Hence, by subtracting, we have
    \begin{equation}
        I_{2^k-1, d} = \begin{cases}
            0 & \text{if $d < 2^k$,} \\
            \vvd_0 p_{2^k} & \text{if $d = 2^k$.}
        \end{cases}
    \end{equation}
    Since $\vvd_0 = 1$, it remains to show that $p_{2^k} = v_1^{2^k}$.
    This is a consequence of the following identity in $\Lambda^{\ZZ_2}$ as a subring of $\ZZ_2[\![x_1, x_2, \dots]\!]$:
    \begin{equation}
        p_{2^k} = \sum_{i \ge 1} x_i^{2^k} = \bigg(\sum_{i\ge 1} x_i\bigg)^{2^k} = v_1^{2^k}.\vspace{-3ex}
    \end{equation}
\end{proof}
\section{Geodesic embeddings}\label{sec:geod}
In this section, we introduce the Kolpakov--Reid--Slavich embedding of arithmetic manifolds, which we will use repeatedly to construct non-cobordant manifolds starting from low dimensions.

Let us fix a totally real number field $k\ne \QQ$, and denote by $R$ its ring of integers; in our construction we will have $k = \QQ(\sqrt{5})$ and $R = \ZZ\!\left[\frac{\sqrt{5} + 1}{2}\right]$.

\begin{defin}
    We say that a hyperbolic manifold $M$ is \emph{good} if it is connected, arithmetic of simplest type, defined over $k$, with admissible quadratic form $f$, and such that its fundamental group $\pi_1(M) < \mathrm O(f)$ is contained in the group of $k$-points $\mathrm O(f, k)$. (Note that a good manifold is necessarily closed, since $k \ne \QQ$.)
\end{defin}

One form of the aforementioned embedding is as follows:
\begin{thm}[{\cite{embedding}, \cite[Theorem~1.2]{spinc}}]\label{thm:embed-krs}
    Let $M$ be a good hyperbolic $n$-manifold with defining form $f$. Then $M$ embeds geodesically into a good hyperbolic $(n+1)$-manifold $M'$ with defining form $f + y^2$, such that the map $i_*\colon \pi_1(M) \inj \pi_1(M')$ is given, on matrices $A \in \mathrm O(f)$, by the direct sum $i_*(A) \coloneqq A \oplus [1]$. In particular, the normal bundle of $M$ in $M'$ is trivial.
\end{thm}
Next, we shall extend this result to hyperbolic manifolds with involutions.
\begin{defin}
    Let $M$ be a hyperbolic manifold equipped with an involution $\tau$. We say that $(M, \tau)$ is \emph{good} if $M$ is good and $\tau$ is a nontrivial isometry \emph{defined over $k$}, that is, represented by an element of $\mathrm{O}(f, k)$.    
\end{defin}

\begin{prop}\label{prop:embed-krs-inv}
    Let $(M, \tau)$ be a good hyperbolic $n$-manifold with involution. Then $M$ embeds geodesically into a good hyperbolic $(n+1)$-manifold with involution $(M', \tau')$, such that:
    \begin{enumerate}[label=(\arabic*)]
        \item $M \in F(\tau')$;
        \item $\nu_{M'}(M)$ is trivial;
        \item $\tau'$ acts on $\nu_{M'}(M)$ by negation;
        \item $\tau$ extends to an involution $\sigma$ of $M'$ defined over $k$.
    \end{enumerate}
\end{prop}
The proof requires a technical lemma:
\begin{lemma}\label{lemma:good-inv}
    Let $\Gamma$ be a subgroup of $\mathrm O(f, k)$ which is arithmetic of simplest type, i.e., commensurable with $\mathrm O(f, R)$, and let $P \in \mathrm O(f, k)$. Then $\Gamma$ and $P^{-1} \Gamma P$ are commensurable.
\end{lemma}
\begin{proof}
    It suffices to show that $\mathrm O(f, R)$ is commensurable with $P^{-1} \mathrm O(f, R)P$. Let $P = A/a$, $P^{-1} = B/b$, with $A, B \in \mathrm M(n, R)$ and $a, b \in \ZZ$. Consider the congruence subgroup $\Gamma(ab) < \mathrm O(f, R)$ consisting of $f$-orthogonal matrices of the form $I + abM$, which is of finite index in $\mathrm O(f, R)$. We shall show that $P$ conjugates it into a subgroup of $\mathrm O(f, R)$.

    Indeed, given $Q = I + abM \in \Gamma(ab)$, we have
    \begin{equation}
        P^{-1} Q P = I + P^{-1}abMP = I + AMB \in \mathrm M(n, R).
    \end{equation}
    Since $P^{-1} Q P \in \mathrm O(f, k)$, the claim follows.
\end{proof}

\begin{proof}[Proof of Proposition~\ref{prop:embed-krs-inv}]
We start by embedding $M$ into a good manifold $M'$ with Theorem~\ref{thm:embed-krs}. Let $T \in \mathrm O(f, k)$ represent the involution $\tau$ and let $S \coloneqq T \oplus [1] \in \mathrm O(f+y^2, k)$. Finally, define $T' \coloneqq \diag(1, \dots, 1, -1) \in \mathrm O(f+y^2, k)$. 

Note that $T'$ and $S$ commute, and that they both normalize $\pi_1(M)$. Without loss of generality, we can assume that $\pi_1(M')$ is normalized by $H \coloneqq \langle T', S\rangle$, by replacing it with the intersection of all its $H$-conjugates. The latter contains $\pi_1(M)$, and is of finite index in $\pi_1(M')$ by Lemma~\ref{lemma:good-inv}. Hence, $T'$ and $S$ define involutions $\tau'$ and $\sigma$ on $M'$: the former fixes $M$ and reverses its normal bundle, the latter extends $\tau$.
\end{proof}

\subsection{Twisting the embedding}
Let $(M, \tau) \inj (M', \tau')$ be an embedding of good manifolds with involutions, as constructed in Proposition~\ref{prop:embed-krs-inv}. There is a natural way to modify $(M', \tau')$ by a cut-and-paste operation along $M$.

\begin{defin}
    Given an embedding $(M, \tau) \inj (M', \tau')$ as in Proposition~\ref{prop:embed-krs-inv}, we define the \emph{twist} $T(M', \tau', M, \tau) \coloneqq (M'', \tau'')$, where $M''$ is obtained by cutting $M'$ along $M$ and re-gluing the boundary components with the isometry $\tau$. The involution $\tau''$ is defined to agree with $\tau'$ on $M''\setminus M \simeq M'\setminus M$, and with $\tau$ on $M$.
\end{defin}
It is not hard to check that
\begin{equation}\label{eq:fixed-twist}
    F(\tau'') = F(\tau) \sqcup F(\tau') \setminus \{M\}.
\end{equation}
As for the normal bundles, we have
\begin{equation}\label{eq:fixed-twist-normal}
    \nu_{M''}(X) \simeq
    \begin{cases}
        \nu_M(X)\oplus \varepsilon^1 & \text{if $X \in F(\tau)$}, \\
        \nu_{M'}(X) & \text{if $X \in F(\tau') \setminus \{M\}$}.
    \end{cases}
\end{equation} 
This construction lets us control the value of the formula of Theorem~\ref{thm:phi-n-sum-Ind} while recursively applying Proposition~\ref{prop:embed-krs-inv}. In order to do so, we show that the twist is arithmetically well behaved.
\begin{lemma}
    The manifold $T(M', \tau', M, \tau) \coloneqq (M'', \tau'')$ is a good hyperbolic manifold with involution.
\end{lemma}
\begin{proof}
    We distinguish two cases. If $M$ separates $M'$, then $M''$ is isometric to $M'$. Up to this isometry, $\tau'' = \sigma\tau'$, where $\sigma$ extends $\tau$. Hence, we only need to show that $\tau''$ is defined over $k$, which is true as $\sigma$ and $\tau'$ are.

    If $M$ does not separate $M'$, then it defines a map $\lambda \colon \pi_1(M') \surj \ZZ_2$ by counting intersections modulo $2$. There is a double cover $N \surj M'$ corresponding to $\ker \lambda$, with deck automorphism $s$. Then $\sigma$ lifts to an involution $\sigma'$ of $N$ defined over $k$, such that $N/\langle s\sigma'\rangle \simeq M''$ (see Figure~\ref{fig:twist}). Therefore, $M''$ and $M'$ are commensurable. Moreover, $\sigma'$ is a lift of the involution $\tau''$ of $M''$, which is hence defined over $k$. It follows that $(M'', \tau'')$ is good.
\end{proof}
\begin{figure}[h]
    \centering
    \begin{tikzpicture}
        % Start the tabular
        \node (fig1) at (0,0) {\includegraphics[width=0.45\textwidth]{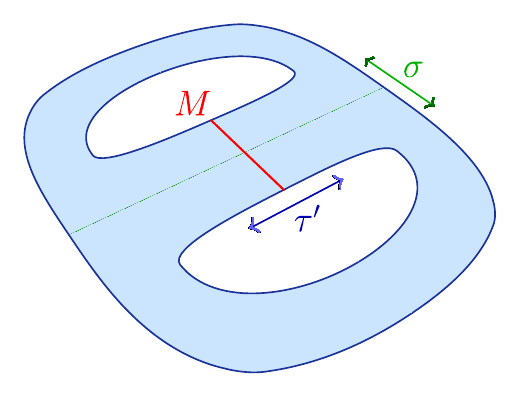}};
        \node (fig2) at (7,0) {\includegraphics[width=0.45\textwidth]{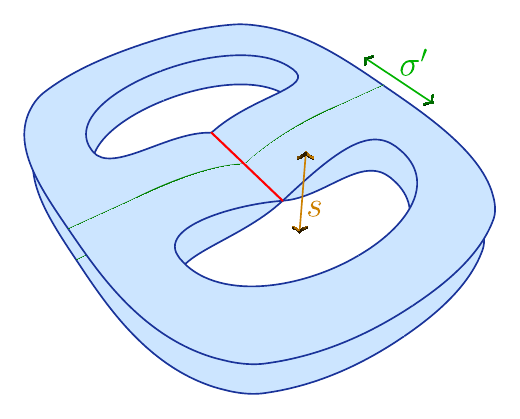}};
        \node (fig3) at (3.5,-5) {\includegraphics[width=0.45\textwidth]{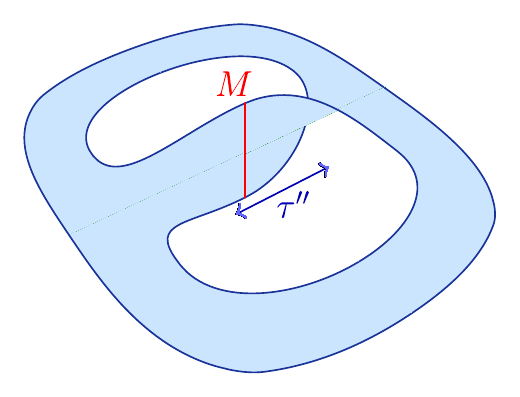}};
        
        % Add curved arrow between twist1 and twist2
        \draw[->, thick, bend left=30] (fig1.60) to (fig2.120);
        
        % Add rotated curved arrow between twist2 and twist3
        \draw[->, thick, bend left=30] ([shift={(0.3,0.2)}]fig2.300) to ([shift={(0.3,0.2)}]fig3.0);
    \end{tikzpicture}
    \caption{Construction of the twist $T(M', \tau', M, \tau)$ when $M$ does not separate $M'$. Clockwise from top left: the manifold $M$ embedded in $M'$; the double cover $N$ with deck automorphism $s$; the twisted quotient $M''$.}
    \label{fig:twist}
\end{figure}
\section{Non-cobordant hyperbolic manifolds}\label{sec:non-cobord}
In this section we show how to construct a hyperbolic $n$-manifold $M^n$ with $\varphi^n[M^n] = 1$, which is therefore non-cobordant, by starting from a simpler manifold and recursively passing to a (possibly twisted) embedding. We then proceed to apply the method for many values of $n$.

\begin{thm}\label{thm:main}
    Let $(M, \tau)$ be a good hyperbolic $k$-manifold with involution, where $k \le n$.
    Suppose that
    \begin{equation}
        I(n, M, \tau) \coloneqq \sum_{d = 0}^{n-1} \sum_{X \in F_d(\tau)} I_{n,d}(w_i(\nu_M X)) = 1.
    \end{equation}
    Then there exists a good hyperbolic $n$-manifold with involution $(\Bar M, \Bar \tau)$ such that $\varphi^n[\Bar M] = 1$. In particular, $\Bar M$ is non-cobordant.
\end{thm}
\begin{proof}
    We prove by induction that, for $k \le m \le n$, there exists a good manifold with involution $(M^m, \tau^m)$ such that
    \begin{equation}\label{eq:proof-induction}
        I(n, M^m, \tau^m) = \sum_{d = 0}^{n-1} \sum_{X \in F_d(\tau^m)} I_{n,d}(w_i(\nu_{M^m} X)) = 1.
    \end{equation}
    (Note that all terms with $d \ge m$ are zero, as $F_d(\tau^m)$ is then empty.) The base case $m = k$ is true by hypothesis. Now, suppose that~(\ref{eq:proof-induction}) holds for $(M^m, \tau^m)$. By Proposition~\ref{prop:embed-krs-inv}, we can embed $M^m$ into a manifold with involution $(M', \tau')$. If $I(n, M', \tau') = 1$, we are done by choosing $(M^{m+1}, \tau^{m+1}) \coloneqq (M', \tau')$.

    Otherwise, let $(M^{m+1}, \tau^{m+1}) \coloneqq T(M', \tau', M^m, \tau^m)$. Then, by~(\ref{eq:fixed-twist}) and~(\ref{eq:fixed-twist-normal}), we have:
\begin{comment}
    \begin{align}
        I(\varphi, M^{m+1}, \tau^{m+1})
        &= \sum_{X \in F(\tau^{m+1})} \varphi\big([\mathbb P(\nu_{M^{m+1}} X\oplus \varepsilon^{n-m})]\big) \\
        &= \sum_{X \in F(\tau')} \varphi\big([\mathbb P(\nu_{M'} X\oplus \varepsilon^{n-m})]\big) \\
        &+ \sum_{X \in F(\tau)} \varphi\big([\mathbb P(\nu_{M^m} X\oplus \varepsilon^{1+n-m})]\big) \\
        &+ \varphi\big([\mathbb P(\nu_{M'} M^m\oplus \varepsilon^{n-m})]\big) \\
        &= I(\varphi, M', \tau') + I(\varphi, M^m, \tau) + \varphi([M^m \times \mathbb{RP}^{n-m}])
    \end{align}
\end{comment}
    \begin{align}
        I(n, M^{m+1}, \tau^{m+1})
        &= I(n, M', \tau') + I(n, M^m, \tau^m) + I_{n,m}(w_i(\nu_{M'} M^m)) \\
        &= I(n, M^m, \tau^m) = 1,
    \end{align}
    where $I_{n,m}(w_i(\nu_{M'} M^m)) = 0$, since $\nu_{M'} M^m$ is trivial and $I_{n,m}$ is homogeneous of degree $m > 0$. (The case $m = k = 0$ can be excluded, as the sum in $I(n,M,\tau)$ would vanish.)

    When $m = n$, we finally obtain a manifold with involution $(\Bar M, \Bar \tau) \coloneqq (M^n, \tau^n)$, such that $\varphi^n [\Bar M] = I(n, \Bar M, \Bar \tau) = 1$.
\end{proof}
\subsection{A starting manifold}

Let $D$ be the first small cover of the hyperbolic right-angled dodecahedron in the left column of~\cite[Table~1]{garrison-scott}. Let us number the facets of the dodecahedron $1$ to $12$, as in~\cite[Figure~3]{garrison-scott}. The reflection in facet $i$ induces an isometry $\tau_i$ of $D$, with $F(\tau_i)$ consisting of isolated points, circles, and surfaces of characteristic $-1$ diffeomorphic to $\mathbb T^2 \# \mathbb{RP}^2$; using SageMath, we classify these components in Table~\ref{tab:dod-small-cover}.

It is also well known that $D$ covers the arithmetic Coxeter simplex $[5,3,4]$, so it is good for $k = \QQ(\sqrt{5})$, and that the isometries $\tau_i$ are conjugate to a generator of the Coxeter group, so they are defined over $k$.

\begin{table}[ht]
    \centering
    \begin{tabular}{cccccc}
        \toprule
        & & \multicolumn{4}{c}{Number of fixed submanifolds of $\tau_i$}
        \\ \cmidrule{3-6}
        Facet $i$ & Color & Points & Circles & Surfaces $X$ & ...with $w_1^2(\nu X) = 1$ \\ \midrule
         $1, 7, 10$ & $(0,0,1)$ & $3$ & $1$ & $3$ & $1$ \\ 
         $2, 5$     & $(0,1,0)$ & $2$ & $4$ & $2$ & $1$ \\ 
         $11$       & $(0,1,1)$ & $5$ & $5$ & $1$ & $0$ \\ 
         $3, 4$     & $(1,0,0)$ & $2$ & $4$ & $2$ & $2$ \\ 
         $12$       & $(1,0,1)$ & $1$ & $7$ & $1$ & $1$ \\ 
         $6$        & $(1,1,0)$ & $1$ & $7$ & $1$ & $0$ \\ 
         $8,9$      & $(1,1,1)$ & $6$ & $2$ & $2$ & $1$ \\
         \bottomrule 
    \end{tabular}
    \caption{Fixed submanifolds for each isometry $\tau_i$ of $D$.}
    \label{tab:dod-small-cover}
\end{table}

By comparing Tables~\ref{tab:Ind} and~\ref{tab:dod-small-cover} we can compute $I(n, D, \tau_i)$ for arbitrary $i = 1, \dots, 12$ and $n \ge 4$. As an example, let $n = 4m\ge 4$. By periodicity (Remark~\ref{rem:periodic}), we have
\begin{equation}
    I_{n,0} = 1,\quad I_{n,1} = 0,\quad I_{n,2} = v_1^2 + v_2.
\end{equation}
The involution $\tau_{11}$ has $5$ isolated fixed points and $0$ fixed surfaces with nontrivial $v_1^2$. Hence, $I(4m, D, \tau_{11}) = 1$ for all $m \ge 1$. Similarly, we can check that $I(4m+1, D, \tau_{1}) = I(4m+2, D, \tau_{1}) = 1$ for all $m \ge 1$.
Hence, we have proved Theorem~\ref{thm:main20}, which we restate here: 
\begin{thm20*} %\label{thm:main2}
    For each $n \ge 4$, $n \not \equiv 3 \pmod 4$, there exists a connected, non-cobordant closed hyperbolic $n$-manifold.
\end{thm20*}

\subsection{The remaining dimensions}\label{sec:higher}
Note that the remaining case $n \equiv 3 \pmod 4$ cannot be handled by starting with a $3$-manifold, since by Proposition~\ref{prop:initial-zeros}, we have $I_{n,0} = I_{n,1} = I_{n,2} = 0$ in that case.

Let us outline a possible general approach to the construction of non-cobordant hyperbolic manifolds in every dimension $n \ne 2^i-1$, using Proposition~\ref{prop:initial-zeros}. Since even dimensions are included in Theorem~\ref{thm:main20}, assume that $n \ne 2^i-1$ is odd; then we can write $n = 2^k(2m+1)-1$ for $k, m \ge 1$. By Proposition~\ref{prop:initial-zeros}, the first nonzero term in $I(n, M, \tau)$ is $I_{n,2^k} = \vv_1^{2^k}$. Because of this, we would have to start at the very least from a manifold with involution $(M_{(k)}, \tau_{(k)})$ of dimension $2^k+1$, such that
\begin{equation}
    \sum_{X \in F_{2^k}(\tau_{(k)})} w_1^{2^k}(\nu X) = 1.
\end{equation}
Given such a manifold, by Proposition~\ref{prop:initial-zeros} and Theorem~\ref{thm:main}, we would have a connected, non-cobordant closed hyperbolic $n$-manifold.

Note that the case $k = 1$ corresponds to $n \equiv 1 \pmod 4$ and to the manifold with involution $(D, \tau_1)$ described above. In general, the cases $k = 1,2,3,4, \dots$ correspond to the sets $\{4m + 1\},\{8m + 3\}, \{16m + 7\}, \{32m + 15\}, \dots (m\ge 1)$, which partition the odd numbers not of the form $2^i-1$.

As a consequence, finding manifolds with involutions $(M_{(1)}, \tau_{(1)}), \dots, (M_{(k)}, \tau_{(k)})$ as above implies the existence of connected, non-cobordant closed hyperbolic $n$-manifolds for all $n \not \equiv -1 \pmod {2^{k+1}}$, $n \not \in \{1, 3, 7, \dots, 2^k-1\}$: a set of dimensions of asymptotic density $1-2^{-k-1}$.

\subsection{Even dimensions}\label{sec:even}
Recall that, for a manifold of dimension $n$, the Stiefel--Whitney number $w_n$ is the Euler characteristic modulo $2$. Hence, in any even dimension $n = 2k$ there is an alternate method for the construction of non-cobordant hyperbolic manifolds: searching for manifolds of odd characteristic. In the literature, the only such closed manifolds we could find are connected sums $\Sigma_g \# \RP^2$ for $g \ge 1$ (where $\Sigma_g$ denotes an orientable genus $g$ surface), of characteristic $1-2g$, and two $4$-manifolds of Euler characteristic $17$, constructed by Ratcliffe and Tschantz~\cite[9]{rt-instantons} by gluing two copies of the right-angled hyperbolic $120$-cell.

On the other hand, we can always ensure that the non-cobordant manifolds from our method have even characteristic, by manipulating the Kolpakov--Reid--Slavich embedding.  Indeed, since the twist operation preserves the Euler characteristic, it suffices to arrange for the embedding $(M, \tau) \inj (M', \tau')$ of Proposition~\ref{prop:embed-krs-inv} to satisfy $\chi(M') \equiv 0 \pmod 2$.

If $M$ separates $M'$, then the two halves of $M'$ are isometric via $\tau'$. It follows that $M'$ is a double, so $\chi(M')$ is even. If instead $M$ does not separate $M'$, then it defines a double cover of $M'$, associated to the map $\pi_1(M) \surj \ZZ_2$ sending $\gamma \mapsto \gamma \cdot D$. In the proof of Proposition~\ref{prop:embed-krs-inv}, we can then replace $M'$ by this double cover after applying Theorem~\ref{thm:embed-krs}, ultimately obtaining a manifold with even Euler characteristic.

If $|\mathcal N_n| = 4$, the two conditions $\varphi^n[M] = 1$ and $w_n(M) = 0$ determine the cobordism class of $M$. More precisely, we can state:

\begin{prop}
    There exists a connected, closed hyperbolic $4$-manifold $M$ of even Euler characteristic, such that $[M] = [\RP^4 \sqcup (\RP^2 \times \RP^2)] \ne 0$.
\end{prop}
\begin{proof}
    The cobordism group $\mathcal N_4$ is generated by $x_4$ and $x_2^2$. By Theorem~\ref{thm:main20}, there exists a connected, closed hyperbolic $4$-manifold $M$ with $\varphi^4[M] = 1$. Hence, $[M]$ is either $[\RP^4]$ or $[\RP^4 \sqcup (\RP^2\times \RP^2)]$. From the above discussion, we may also assume that $\chi(M)$ is even, which implies  $[M] = [\RP^4 \sqcup (\RP^2\times \RP^2)]$.
\end{proof}
\begin{remark}
    In higher even dimensions $n$, there are at least three linearly independent generators of $\mathcal N_n$, that is, $x_n$, $x_{n-2}x_2$ and $x_2^{n/2}$. Thus, $\varphi^n$ and $w_n$ do not suffice to determine the cobordism class.
\end{remark}
\clearpage
\setlength\bibitemsep{1ex}
\printbibliography[heading=bibintoc, title={References}]
\end{document}